\documentclass[a4paper,11pt]{amsart}
\usepackage[utf8]{inputenc}
\usepackage{lmodern}
\usepackage{latexsym}
\usepackage{a4wide}
\usepackage{graphics}
\usepackage{amsmath}
\usepackage{amsthm}
\usepackage{amssymb}
\usepackage{footnote}
\usepackage{bm}
\usepackage{enumerate}
\usepackage{csquotes}
\usepackage{pgf,tikz}
\usepackage{dsfont,mathrsfs}
\usepackage{color}

\usepackage{hyperref}
\usepackage{cleveref}
\hypersetup{
	linktocpage=true,
	colorlinks=true,
	linkcolor=blue!70!black,
	citecolor=orange!40!red,
	urlcolor=magenta!80!black,
}

\numberwithin{equation}{section}

\newtheorem{theorem}{Theorem}[section]
\newtheorem{corollary}[theorem]{Corollary}
\newtheorem{lemma}[theorem]{Lemma}
\newtheorem{proposition}[theorem]{Proposition}
\theoremstyle{definition}

\newtheorem{remark}[theorem]{Remark}

\newcommand{\R}{\mathbb{R}}	
\newcommand{\N}{\mathbb{N}} 
\newcommand{\dx}{\,\mathrm{d}x}	
\newcommand{\dy}{\,\mathrm{d}y}	
\newcommand{\dxdy}{\,\mathrm{d}x \mathrm{d}y}
\newcommand{\dxdt}{\,\mathrm{d}x \mathrm{d}t}
\newcommand{\ds}{\,\mathrm{d}S}	
\renewcommand{\d}{\,\mathrm{d}}
\newcommand{\norm}[1]{\left\lVert #1 \right\lVert}
\newcommand{\abs}[1]{\left| #1 \right|}
\newcommand{\sub}{\subseteq}
\newcommand{\weak}{\rightharpoonup}

\newcommand{\V}{\mathscr{V}}
\newcommand{\loc}{\textup{loc}}
\newcommand{\rad}{\textup{rad}}

\DeclareMathOperator{\dive}{\mathrm{div}}

\newenvironment{bvp}{\left\{\begin{aligned}  }{\end{aligned}\right.}

\newcommand{\Ds}{\mathcal{D}^{s,2}(\R^N)} 
\newcommand{\Dsr}{\mathcal{D}^{s,2}_{\rad}(\R^N)} 

\newcommand{\eps}{\varepsilon}

\title[]{Local multiplicity for fractional linear equations with Hardy potentials}

	\author[E. Mainini]{Edoardo Mainini}
\address{Edoardo Mainini
	\newline \indent Dipartimento di Ingegneria Meccanica, Energetica, Gestionale e dei Trasporti
	\newline \indent Universit\`a di Genova
	\newline\indent  via all'Opera Pia 15,  16145 Genova, Italy}
\email{edoardo.mainini@unige.it}

\author[R. Ognibene]{Roberto Ognibene}
\address{Roberto Ognibene
	\newline \indent Dipartimento di Matematica
	\newline \indent Universit\`a di Pisa
	\newline\indent  Largo Bruno Pontecorvo, 5, 56127 Pisa, Italy}
\email{roberto.ognibene@dm.unipi.it}

\author[B. Volzone]{Bruno Volzone}
\address{Bruno Volzone
	\newline \indent Dipartimento di Matematica
	\newline \indent Politecnico di Milano
	\newline\indent  Piazza Leonardo da Vinci 32, 20133 Milano, Italy}
\email{bruno.volzone@polimi.it}

\begin{document}

\keywords{Fractional Schr\"odinger equations, uniqueness of solutions, Hardy potential}

\subjclass{ 35A02,  
	35R11, 
	35B40, 
	35J75  
}

\begin{abstract}
	We exhibit existence of non-trivial solutions of some fractional linear Schr\"odinger equations which are radial and vanish at the origin. This is in stark contrast to what happens in the local case. We also prove analogous results in the presence of a Hardy potential.
\end{abstract}
	
\maketitle

\tableofcontents

\section{Introduction}\label{sec:introduction}

\subsection{Setting}

 For $N\in\N$ and $s\in(0,\min\{N/2,1\})$, we introduce the fractional Laplacian  in $\R^N$, defined as follows for $u\in C_c^\infty(\R^N)$
\[
	(-\Delta)^s u(x):=C_{N,s}\,\mathrm{p.v.}\int_{\R^N} \frac{u(x)-u(y)}{|x-y|^{N+2s}}\dy,\quad x\in\R^N.
\]
Here,
\[
C_{N,s}:=2^{2s}\pi^{-N/2}\frac{\Gamma(\tfrac{N+2s}{2})}{|\Gamma(-s)|}
\]
is chosen in such a way that
\[
	(-\Delta)^s u(x)\to -\Delta u(x)\quad\text{as }s\to 1^-~\text{for all }x\in\R^N.
\]
Since, by simple integration by parts, one can see that
\[
	\int_{\R^N}(-\Delta)^s u v\dx=\frac{C_{N,s}}{2}\int_{\R^{2N}}\frac{(u(x)-u(y))(v(x)-v(y))}{|x-y|^{N+2s}}\dxdy=\int_{\R^N}u(-\Delta)^s v\dx
\]
for all $u,v\in C_c^\infty(\R^N)$, one is led to consider it as a natural scalar product in $C_c^\infty(\R^N)$ related to $(-\Delta)^s$ and to work in the corresponding complete space. Hence, we define $\mathcal{D}^{s,2}(\R^N)$ as the completion of $C_c^\infty(\R^N)$ with respect to the norm 
\[
	\norm{\cdot}_{\mathcal{D}^{s,2}(\R^N)}:=\sqrt{(\cdot,\cdot)_{\mathcal{D}^{s,2}(\R^N)}}
\]
corresponding to the scalar product
\[
	(u,v)_{\mathcal{D}^{s,2}(\R^N)}:=\frac{C_{N,s}}{2}\int_{\R^{2N}}\frac{(u(x)-u(y))(v(x)-v(y))}{|x-y|^{N+2s}}\dxdy.
\]
Therefore, within this framework, it is possible to extend the definition of the fractional Laplacian $(-\Delta)^s$ (in a weak sense) to functions in $\mathcal{D}^{s,2}(\R^N)$. In view of the fractional Sobolev inequality, asserting that
\begin{equation*}
	S_{N,s}\norm{u}_{L^{2^*_s}(\R^N)}\leq \norm{u}_{\mathcal{D}^{s,2}(\R^N)}\quad\text{for all }u\in\mathcal{D}^{s,2}(\R^N),
\end{equation*}
where $2^*_s:=2N/(N-2s)$ and $S_{N,s}>0$, we have that $\mathcal{D}^{s,2}(\R^N)$ can be characterized as a concrete functional space, that is
\[
	\mathcal{D}^{s,2}(\R^N)=\{u\in L^{2^*_s}(\R^N)\colon \norm{u}_{\mathcal{D}^{s,2}(\R^N)}<\infty\}.
\]
On the other hand, it is well known that a Hardy inequality holds true in this fractional setting (see \cite{Herbst1977}), with optimal constant
\[
	\Lambda_{N,s}:=2^{2s}\left(\frac{\Gamma(\tfrac{N+2s}{4})}{\Gamma(\tfrac{N-2s}{4})}\right)^2.
\]
More precisely, there holds
\begin{equation}\label{eq:hardy}
	\Lambda_{N,s}\int_{\R^N}\frac{u^2}{|x|^{2s}}\dx\leq \norm{u}_{\mathcal{D}^{s,2}(\R^N)}^2\quad\text{for all }u\in \mathcal{D}^{s,2}(\R^N),
\end{equation}
so that the following additional characterization holds true
\begin{equation*}
	\Ds=\left\{u\in L^1_\loc(\R^N)\colon \norm{u}_{\Ds}^2+\int_{\R^N}\frac{u^2}{|x|^{2s}}\dx<\infty\right\}.
\end{equation*}

\subsection{Main results}  We consider the following equation
\begin{equation}\label{eq:main}
			(-\Delta)^su-\frac{\lambda}{|x|^{2s}} u =\V u \quad\text{in }B_1,
\end{equation}
in a weak sense, that is $u\in\Ds$ and
\begin{equation}\label{eq:main_weak}
	(u,\varphi)_{\mathcal{D}^{s,2}(\R^N)}-\lambda\int_{\R^N}\frac{u\varphi}{|x|^{2s}}\dx=\int_{\R^N}\V u\varphi\dx\quad\text{for all }\varphi\in C_c^\infty(B_1),
\end{equation}
where for any $r>0$ we denote $B_r:=\{x\in\R^N\colon |x|<r\}$. Moreover,  $\lambda<\Lambda_{N,s}$ and $\V\in C^1(B_1\setminus\{0\})$ is a  potential satisfying the following growth condition, which essentially says that, at the origin, it is negligible with respect to the Hardy potential:
\begin{equation}\label{eq:G}\tag{G}
	|\V(x)|+|x\cdot\nabla \V(x)|\leq C_{\V} |x|^{-2s+\eps},\quad\text{as }|x|\to 0,\quad\text{for some }C_{\V}>0\,,\eps>0.
\end{equation}
Our focus is on { radial solutions} of \eqref{eq:main}, therefore we introduce the space   $\mathcal{D}^{s,2}_{\rad}(\R^N)$ of radial functions in $\mathcal{D}^{s,2}(\R^N)$.  
Hence, we call $u$ a \emph{radial   weak   solution} of \eqref{eq:main} if  $u\in \mathcal{D}^{s,2}_{\rad}(\R^N)$ and $u$ satisfies \eqref{eq:main_weak}. We also introduce the following quantity, that will play a key role in our main result. We let
\begin{equation}\label{eq:gamma_lamda}
	\gamma_\lambda:=-\frac{N-2s}{2}+\sqrt{\left(\frac{N-2s}{2}\right)^2+\mu_1^s(\lambda)},
\end{equation}
being $\mu_1^s(\lambda)$   the first eigenvalue of    a spherical problem associated to the fractional Laplacian that we shall describe in detail in \Cref{sec:nonlocal_behvaior}, see \eqref{spherical} and \eqref{eq:mu_1} below. The value $\gamma_\lambda$ appears  in the local expansion of solutions of \eqref{eq:main} near the origin. Indeed,  in \Cref{sec:nonlocal_behvaior} we will recall a result by Fall and Felli \cite{Fall2014}, entailing that every radial solution $u$ to \eqref{eq:main} is continuous in $B_1\setminus\{0\}$ and such that the limit $\lim_{x\to 0}|x|^{-\gamma_\lambda}u(x)$ exists finite. In Lemma \ref{lemma:Psi_gamma} we will also provide a characterization of $\gamma_\lambda$ in terms of the inverse of an explicit function which naturally appears in the ground state representation for the fractional Laplacian (see e.g. \cite{Frank2008,Dipierro}).   
We are now ready to state our main result, which is about local non-uniqueness of radial solutions.
\begin{theorem}\label{thm:main_1}
	Let $\lambda<\Lambda_{N,s}$ and let $q>1$. Then, there exists a   radial   potential
	\begin{equation*}
		\V\in C^{1,\alpha}_{\loc}(B_1\setminus\{0\})\cap L^q(B_1)\quad\text{for some }\alpha\in(0,1)
	\end{equation*}
	satisfying \eqref{eq:G}, such that $\V>0$ in $B_1\setminus\{0\}$ and   satisfying the following property: equation \eqref{eq:main} admits a nontrivial radial weak solution $u\in\Dsr$  
	 such that
	\begin{equation}\label{limitcondition}
		\lim_{|x|\to 0}|x|^{-\gamma_\lambda}u(x)=0,
	\end{equation}
	where $\gamma_\lambda$ is as in \eqref{eq:gamma_lamda}. Moreover, $u\equiv 0$ in $\R^N\setminus B_R$ for some $R\geq 1$, and $u\not\equiv 0$ in $\Omega$ for any non-empty open set $\Omega\sub B_R$.
\end{theorem}

  Further estimates about the behavior at the origin of the potential $\V$ in the above theorem can also be obtained, depending on $\lambda, s, N$, as will be seen from the proof.  
 In case of absence of the Hardy potential, i.e. $\lambda=0$, we state the following variant.  
\begin{theorem}\label{thm:main_2}
	There exists a   radial   potential   $\V \in C^1(B_1\setminus\{0\})\cap C^0(B_1)$ satisfying \eqref{eq:G}, such that $\V>0$ in $B_1\setminus\{0\}$ and satisfying the following property: there is a nontrivial radial weak solution $u\in \Dsr$ to  
	\begin{equation*}
		(-\Delta)^s u=\V u\quad\text{in }B_1
	\end{equation*}
	 such that
	\begin{equation*}
		u(0)=0.
	\end{equation*}
	Moreover, $u\equiv 0$ in $\R^N\setminus B_R$, for some $R\geq 1$, and $u\not\equiv 0$ in $\Omega$ for any non-empty open set $\Omega\sub B_R$. Finally, 
	\begin{itemize}
		\item if $N\geq 6s$ and $0<\alpha<4s/(N-2s)$, we can take $\V\in C^{0,\alpha}_{\loc}(B_1)$;
		\item if $N<6s$, we can take $\V\in C^1(B_1)$.
	\end{itemize}
\end{theorem}

We observe that our main result \Cref{thm:main_1} stands in stark contrast with respect to other situations   in which uniqueness of solutions is known. The first evident discrepancy is with the classical case $s=1$: indeed, in this case the local condition
\[
	\lim_{x\to 0}|x|^{-\gamma_\lambda} u(x)=0,
\]
which  boils down to $u(0)=0$ if $\lambda=0$, forces any radial solution of
\[
	-\Delta u-\frac{\lambda}{|x|^2}u=\V u\quad\text{in }B_1
\]
to vanish identically. We will check this in \Cref{cor:local_radial}.

Concerning the fractional case, in the breakthrough \cite{FLS} the authors proved that for bounded radially decreasing $\V$ there exists at most one radial (variational) solution of
\[
	(-\Delta )^s u=\V u \quad\text{in }\R^N
\]
which vanishes at the origin,   i.e., the trivial solution.  

Also in the nonlinear case uniqueness results are available. For instance, uniqueness of ground states for some particular (power-type) nonlinear problems in the whole $\R^N$ has been proved in \cite{FL} (for $N=1$) and \cite{FV,FLS,CGHMV} (in higher dimension). Finally, we mention \cite{fall2024nonradial,delatorre2024uniqueness}, where the authors proved uniqueness results for fractional nonlinear problems in balls, with homogeneous Dirichlet exterior conditions.  

On the contrary, with our main results we prove that there are cases (namely, potentials) in which (local) uniqueness of radial solutions of \eqref{eq:main} does not hold: in fact, our approach consists in building such counterexamples.   We stress however that our results do not cover problems posed in the whole space, and in particular the validity of the result in \cite{FLS} in absence of the monotonicity assumption on $\V$ remains an open question.    

 Let us briefly describe our approach. First of all, in \Cref{sec:multiplicity}, by making use of abstract critical points theory, we check the existence of infinitely many \emph{distinct radial} solutions to the nonlinear problem
\begin{equation}\label{nonlinearp}
	\begin{bvp}
		(-\Delta)^su-\frac{\lambda}{|x|^{2s}}u&= |u|^{p-2}\,u&&\mbox{in $B_1$}, \\
		u&=0, &&\mbox{in }\R^N\setminus B_1
	\end{bvp}
\end{equation}
for $2<p<2^*_s$. In view of the results established in \cite{Fall2014} (which are   crucial, mostly in the case $\lambda\neq 0$), we are able to describe the local behavior near the origin of radial solutions of \eqref{nonlinearp}, and this is done \Cref{sec:nonlocal_behvaior}. In particular, 
 such solutions are $C^1(B_1)$ if $\lambda=0$  and so their value at $0$ is finite in such case. We complete the proof of the main theorems in \Cref{sec:main}. More precisely, by taking into consideration two distinct nontrivial radial solutions of \eqref{nonlinearp} and by suitably scaling them and taking the difference, we produce a radial solution $w$ of \eqref{eq:main} for some suitable  radial potential $\V$ possessing  the properties displayed in \Cref{thm:main_1} and \Cref{thm:main_2} (thanks to the regularity results of \Cref{sec:nonlocal_behvaior}). By construction such a solution satisfies $\lim_{x\to 0}|x|^{-\gamma_\lambda} w(x)=0$.    Moreover, it will be shown to be  non-trivial by invoking the the fact that if a function vanishes on an open set together with its fractional Laplacian then it must be identically $0$. 
   The latter unique continuation property, not holding for the classical Laplacian, will be recalled in  \Cref{thm:UCP} below.  

\section{Infinitely many radial solutions to nonlinear problems}\label{sec:multiplicity}
Let $\lambda<\Lambda_{N,s}$ and $1<p<2_s^*$, $p\neq 2$. We consider the  subcritical problem \eqref{nonlinearp} for the fractional Laplacian with Hardy potential,
with the aim of proving existence of infinitely many radial solutions. In order to give the definition of solution (in a weak sense), we first introduce the space 
\[
H^s_0(B_1):=\{u\in \mathcal D^{s,2}(\R^N): u\equiv 0 \mbox{ in $\R^N\setminus B_1$}\}. 
\]
 
Clearly, 
\[
H^s_0(B_1)\subset H^s(\mathbb R^N):=\mathcal D^{s,2}(\R^N)\cap L^2(\mathbb R^N).
\]
 
Moreover, we define
$
H^s_{0,\rad}(B_1):=\Dsr\cap H^s_0(B_1).
$
Now, we say that $u\in \Ds$ is a \emph{radial solution} of \eqref{nonlinearp} if $u\in H^s_{0,\rad}(B_1)$ and
\[
(u,\varphi)_{\Ds}-\lambda\int_{B_1}\frac{u\varphi}{|x|^{2s}}\dx=\int_{B_1}|u|^{p-2}u\varphi\dx\quad\text{for all }\varphi\in C_c^\infty(B_1).
\]

\begin{proposition}\label{multiplicity}
	Let $\lambda<\Lambda_{N,s}$ and let $1<p< 2_s^*$, $p\neq 2$. Then there exist infinitely many distinct radial  solutions of \eqref{nonlinearp}.  
\end{proposition}
\begin{proof}
	We endow the Hilbert space $H^s_{0,\,\rad}(B_1)$, which is a closed subspace of $\mathcal D^{s,2}_{\rad}(\R^N)$, with the equivalent norm $\norm{\cdot}_{s,\lambda}$ induced by the scalar product
	\[
	(u,v)_{s,\lambda}:=(u,v)_{\Ds}-\lambda\int_{\R^N}\frac{uv}{|x|^{2s}}\dx,\quad u,v\in H^s_{0,\rad}(B_1),
	\]
	and we define the even functional $\mathcal I:H^{s}_{0,\,\rad}(B_1)\to\mathbb R$ by
	\[
	\mathcal I(u):=\frac1p\int_{B_1} |u|^p\dx.
	\]
	Indeed, $\norm{\cdot}_{s,\lambda}$ is equivalent to $\norm{\cdot}_{\Ds}$ in view of the Hardy inequality \eqref{eq:hardy}, being $\lambda<\Lambda_{N,s}$.
	
	We preliminarily recall that
	\begin{equation}\label{eq:compact}
		H^s_{0,\rad}(B_1)\hookrightarrow L^q(B_1)\quad  \text{compactly for }1<q<2^*_s
	\end{equation}
	and from this we immediately derive continuity of $\mathcal{I}$   with respect to the weak convergence in $ H^s_{0,\rad}(B_1)$ . Let us now prove that $\mathcal{I}\in C^1(H^s_{0,\rad}(B_1))$. Let $(u_n)_{n\in\N}\sub H^s_{0,\rad}(B_1)$ and $u\in H^s_{0,\rad}(B_1)$ be such that
	\[
	u_n\weak u\quad\text{weakly in }H^s_{0,\rad}(B_1)~\text{as }n\to\infty.
	\]
	Since
	\begin{multline*}
		\sup_{\substack{\varphi\in H^s_{0,\rad}(B_1) \\ \|\varphi\|_{s,\lambda}\le 1}}\int_{\R^N}||u_n|^{p-2}u_n-|u|^{p-2}u||\varphi|\dx \\
		\le\left(\int_{\R^N}||u_n|^{p-2}u_n-|u|^{p-2}u|^{\frac{p}{p-1}}\dx\right)^{\frac{p-1}{p}}\sup_{\substack{\varphi\in H^s_{0,\rad}(B_1) \\ \|\varphi\|_{s,\lambda}\le 1}}\left(\int_{\R^N} |\varphi|^p\dx \right)^\frac1p,
	\end{multline*}
	where the supremum in the right hand side is finite again due to \eqref{eq:compact}, we deduce that 
	\begin{equation}\label{dualnorm}
		\lim_{n\to+\infty}\sup_{\|\varphi\|_{s,\lambda}\le 1}\int_{\R^N}||u_n|^{p-2}u_n-|u|^{p-2}u||\varphi|=0,
	\end{equation}
	as $u_n \to u$ in $L^p(\R^N)$ implies $|u_n|^{p-2}u_n\to |u|^{p-2}u$ in $L^{p/(p-1)}(\R^N)$.
	If we now denote by $\mathcal I'(v)\in (H^{s}_{0,\,\rad}(B_1))'$ the Fr\'echet differential of $\mathcal I$ at $v\in H^s_{0,\,\rad}(B_1)$, we have that
	\[
	\mathcal{I}'(u)[\varphi]=\int_{B_1}|v|^{p-2}v\varphi\dx\quad\text{for all }\varphi\in H^s_{0,\rad}(B_1).
	\]
	Hence, \eqref{dualnorm} shows that the operator 
	\begin{equation}\label{oper}\begin{aligned}
		H^s_{0,\,\rad}(B_1) &\to (H^s_{0,\,\rad}(B_1))' \\
		v&\mapsto \mathcal{I}'(v)
	\end{aligned}
	\end{equation}
	is compact, since it shows $\mathcal I'(u_n)\to\mathcal I'(u)$ in $(H^s_{0,\,\rad}(B_1))'$.
	In particular, $\mathcal I\in C^1(H^s_{0,\,\rad}(B_1))$.
	
	Now, let $\mathcal I_{s,\lambda}$ be the restriction of $\mathcal I$ to 
	$$S_{s,\lambda}:=\{u\in H^{s}_{0,\,\rad}(B_1): \|u\|_{s,\lambda}=1\}$$
	and let $\mathcal I'_{s,\lambda}(u)$ be the differential of $\mathcal I_{s,\lambda}$ at $u\in S_{s,\lambda}$, i.e., the restriction of the differential $\mathcal I'(u)$ at $u$ on the tangent space at $S_{s,\lambda}$ at $u$ :
\[
T_{u}(S_{s,\lambda})=\left\{\psi\in H^{s}_{0,\,\rad}(B_1):\,( u,\psi)_{s,\lambda}=0 \right\}.
\]
Therefore, for any test function $\varphi\in H^{s}_{0,\,\rad}(B_1)$ we find
\begin{equation}\label{tangential}
		\mathcal I'_{s,\lambda}(u)[\varphi_{tan}]=\mathcal I'(u)[\varphi]-\mathcal I'(u)[u]( u,\varphi)_{s,\lambda},\qquad\mbox{$\varphi\in H^{s}_{0,\,\rad}(B_1)$},
	\end{equation}
where $\varphi_{tan}:=\varphi-( u,\varphi)_{s,\lambda}u\in T_{u}(S_{s,\lambda})$ is the \emph{tangential} component of $\varphi$.
	
	  The proof will be concluded by invoking the infinite-dimensional Lusternik-Schnirelman theorem from \cite[Theorem 8.10, Remark 8.12]{R}. In order to apply it, we need to verify the validity of a Palais-Smale condition. In particular, we shall check that $\mathcal I_{s,\lambda}$ satisfies the $(PS)_c$ condition for every level $c>0$, that is, if $c>0$ and $( u_n)_{n\in\mathbb N}\subset S_{s,\lambda}$ is a sequence such that 
	\begin{equation}\label{PSC}\mathcal I_{s,\lambda}( u_n)\to c\qquad\mbox{and} \qquad \mathcal I_{s,\lambda}'( u_n)\to0\;\; \mbox{in}\;\; (T_{u}(S_{s,\lambda}))'\qquad\mbox{as $n\to+\infty$},\end{equation} then $( u_n)_{n\in\mathbb N}$ admits a subsequence that converges in $H^s_{0,\rad}(B_1)$.
	Suppose therefore that $c\in(0,+\infty)$ and  that $( u_n)_{n\in\mathbb N}\subset S_{s,\lambda}$ is a sequence that satisfies \eqref{PSC}.   Since 
	 $(u_n)_{n\in\N}$ is bounded in $H^s_{0,\rad}(B_1)$, by recalling \eqref{eq:compact} there exists $u\in H^s_{0,\rad}(B_1)$ such that up to extracting a subsequence $u_n$ converge to $u$ weakly in    $H^s_{0,\rad}(B_1)$   and
		{strongly in }$L^q(  B_1 )$ {for every }$1<q<2^*_s$
	as $n\to\infty$. Thanks to this and to the fact that $\mathcal{I}(u_n)=\mathcal{I}_{s,\lambda}(u_n)\to c$ as $n\to\infty$, we deduce 
	\begin{equation}\label{eq:c}
		\mathcal{I}(u)=c
	\end{equation}  
	and 
	\begin{equation}\label{eq:pc}
		\lim_{n\to+\infty}\mathcal{I}'(u_n)[u_n]=\mathcal I'(u)[u]=p\mathcal I(u)=pc
	\end{equation}
	Notice that from \eqref{tangential} we have
	\[
	\mathcal I'_{s,\lambda}(u_{n})[\varphi_{tan}]=\mathcal I'(u_{n})[\varphi]-\mathcal I'(u_{n})[u_{n}]( u_{n},\varphi)_{s,\lambda}
	\]
	thus from  \eqref{PSC}
	\[
	\lim_{n\to+\infty}\mathcal I'(u_n)[\varphi]=\lim_{n\to+\infty}\mathcal I'(u_n)[u_n]( u_n,\varphi)_{s,\lambda}.
	\]
	 Hence from the compactness of the operator in \eqref{oper}, along with  \eqref{eq:pc}  we obtain that 
	\[
	\mathcal{I}'(u)[\varphi]=\lim_{n\to+\infty}\mathcal I'(u_n)[\varphi]=\lim_{n\to+\infty}\mathcal I'(u_n)[u_n]( u_n,\varphi)_{s,\lambda}=pc(u,\varphi)_{s,\lambda}\quad\text{for all }\varphi\in H^s_{0,\rad}(B_1). 
	\]
	Hence, by choosing $\varphi=u$ and recalling \eqref{eq:c}-\eqref{eq:pc} we obtain that 
	\[
	\norm{u}_{s,\lambda}^2=\frac{1}{pc}\mathcal{I}'(u)[u]=\frac{1}{c}\mathcal{I}(u)=1=\norm{u_n}_{s,\lambda}^2,
	\]
	thus implying that $u_n\to u$ strongly in $H^s_{0,\rad}(B_1)$, so that $(PS)_c$ hods true. Therefore, since $\mathcal I\in C^1(H^s_{0,\,\rad}(B_1))$ is an even functional and since $\mathcal I_{s,\lambda}$ satisfies $(PS)_c$ for every $c>0$,   \cite[Theorem 8.10]{R} shows the existence of infinitely many critical points for $\mathcal I_{s,\lambda}$ (by \cite[Remark 8.12]{R}, the validity of $(PS)_c$ at every level $c>0$, excluding the level $c=0$, is enough to conclude).  
	
	By \eqref{tangential}, each of the obtained nontrivial critical points $u\in H^s_{0,\,\rad}(B_1)$ of $\mathcal I_{s,\lambda}$ satisfies
	$$( u,\varphi)_{s,\lambda}=\delta\mathcal I'(u)[\varphi]\quad\text{for every }\varphi\in H^s_{0,\,\rad}(B_1)$$
	where
	\[
	\delta:=\left(\int_{B_1}|u|^p\dx\right)^{-1},
	\]
	that is
	\begin{equation}\label{radialtest}
		(u,\varphi)_{\mathcal D^{2,s}(\R^N)}-\lambda\int_{\R^N}\frac{u\varphi}{|x|^{2s}}=\delta\int_{\R^N}|u|^{p-2}u\varphi \qquad\mbox{for every $\varphi\in H^s_{0,\,\rad}(B_1)$}.
	\end{equation}In particular, 
	the relation \eqref{radialtest}
	shows that the restriction to $H^s_{0,\,\rad}(B_1)$ of the functional $\mathcal F:H^s_{0}(B_1)\to \R$ defined by $$\mathcal F(\zeta):=\frac12\|\zeta\|^2_{s,\lambda}-\frac \delta p \int_{\R^N}|\zeta|^p$$ has a critical point at $u$.
	Taking advantage of the fact that functional $\mathcal F$ is invariant under the action of the group of rotations, i.e., $\mathcal F(\zeta_{\mathsf R})=\mathcal F(\zeta)$ for any rotation matrix $\mathsf R$, where $\zeta_\mathsf R(x):=\zeta(\mathsf Rx)$,
	\eqref{radialtest} holds in fact  for every $\varphi\in C_c^\infty(B_1)$, by the principle of symmetric criticality, see \cite[Theorem 2.2]{KO}.
	Therefore $\delta^{\frac{1}{2-p}} u\in H^{s}_{0,\,\rad}(B_1)$ is a radial solution of \eqref{nonlinearp}.
\end{proof}

\begin{remark}\rm
 
The result of Proposition \ref{multiplicity} has been established for radial solutions to 
\[
(-\Delta)^su+u=|u|^{p-2}u\quad\mbox{in $\R^N$}
\]
in \cite[Theorem 6.1]{NS}, for $2<p<2_s^*$, by using the same techniques.
\end{remark}

\begin{remark}\rm
The case of the classical Laplace operator can be treated in the same way, for instance obtaining 
the existence of infinitely many radial solutions for the Dirichlet problem
\begin{equation}\label{struwe}\left\{\begin{array}{ll}-\Delta u =|u|^{p-2}u\quad&\mbox{in $B_1$}\\
u=0\quad&\mbox{on $\partial B_1$},\end{array}\right.\end{equation}
 for $2<p<\tfrac{2N}{N-2}$, $N\ge 3$. Radial solutions to $-\Delta u =|u|^{p-2}u$ are governed by the ODE in the radial variable $r$
\begin{equation*}\label{radialode}
-u''-\frac{N-1}{r}u'=|u|^{p-2}u, \qquad u'(0)=0.
\end{equation*}
which
admits a unique solution on $[0,+\infty)$ for every prescribed initial datum $u(0)\in \mathbb R$, see \cite[Theorem 4]{HW}, see also \cite{S}: it is clearly the trivial solution if $u(0)=0$, and it has infinitely many sign changes otherwise. Note that such solutions are (up to sign change) a rescaling of each other. Indeed, if $u$ is a solution to $-\Delta u =|u|^{p-2}u$ in $\mathbb R^N$, then also $u_\delta(x):=\delta u(\delta^{\frac{p-2}{2}}x)$ is, for every $\delta>0$. Moreover, it is standard to check by ODE methods that any radial solution to \eqref{struwe} can be extended in the whole $\mathbb R^N$. In particular, radial solutions to the Dirichlet problem \eqref{struwe} are rescalings each other.
This is a crucial difference with the nonlocal case, since the unique continuation property of the fractional Laplacian prevents two distinct $\mathcal D^{s,2}(\mathbb R^N)$ radial solutions of \eqref{nonlinearp} to be the rescaling of each other (as we shall detail in the proofs of  \Cref{sec:main}),  allowing us to prove our main results.
\end{remark}

\section{Behavior of radial solutions near the origin}

In this section, we collect some known results concerning the local behavior of solutions to a certain class of (local and nonlocal) second-order elliptic PDEs and we then derive the local behavior for radial solutions. We point out that the content of this section is essentially contained in \cite{FFT} and \cite{Fall2014}; however, for the sake of clarity, we state the results in our framework and we provide some proofs. It is nowadays well known that the local behavior of solutions to PDEs, in many cases, is governed by eigenvalues and eigenfuctions of some suitable corresponding spherical operator.    In the the local case $s=1$ the spherical eigenvalue problem is posed $\mathbb{S}^{N-1}$, in the fractional case it is posed in $\mathbb{S}^N_+$, the upper half-sphere in one more dimension, as we next describe. This difference
among   the fractional and non-fractional framework is the core of the question of (local) uniqueness/non-uniqueness of radial solutions. We split the analysis in the two cases $s=1$ and $0<s<1$.

\subsection{\texorpdfstring{The case $s=1$}{The case s=1}}

Here in this section,   concerning the case $s=1$, we let $N> 2$ and we fix  
\begin{equation*}
	\lambda<\left(\frac{N-2}{2}\right)^2
\end{equation*}
and $\V\in L^\infty_{\loc}(B_1\setminus\{0\})$ satisfying
\begin{equation}\label{eq:G_1}\tag{G1}
	|\V(x)|\leq C_\V|x|^{-2+\eps}\quad\text{as }|x|\to 0,~\text{for some }C_\V,\eps>0
\end{equation}
and we consider weak solutions of the following equation
\begin{equation}\label{eq:local}
	-\Delta u-\frac{\lambda}{|x|^2}u=\V u\quad\text{in }B_1,
\end{equation}
that is, functions $u\in H^1(B_1)$ such that
\begin{equation*}
	\int_{B_1}\left(\nabla u\cdot\nabla\varphi-\frac{\lambda}{|x|^2}u\varphi\right)\dx=\int_{B_1}\V u\varphi\dx\quad\text{for all }\varphi\in C_c^\infty(B_1).
\end{equation*}
The aim of the present section is to recall what is the local behavior of solutions to \eqref{eq:local} and to check that results analogous to \Cref{thm:main_1} and \Cref{thm:main_2} cannot hold when $s=1$. Let us consider the following eigenvalue problem on the sphere
\begin{equation}\label{eq:eigen_local}
	(-\Delta_{\mathbb{S}^{N-1}}-\lambda)\phi=\mu\phi,\quad\text{in }\mathbb{S}^{N-1},
\end{equation}
to be understood in the sense that $\phi \in H^1(\mathbb{S}^{N-1})$ and 
\begin{equation*}
	\int_{\mathbb{S}^{N-1}}(\nabla_{\mathbb{S}^{N-1}}\phi\cdot\nabla_{\mathbb{S}^{N-1}}\psi-\lambda\phi\psi)\ds=\mu\int_{\mathbb{S}^{N-1}}\phi\psi\ds\quad\text{for all }\psi\in H^1(\mathbb{S}^{N-1}).
\end{equation*}
It is evident that this is a translation by $-\lambda$ of the eigenvalue problem for the Laplace-Beltrami operator on the sphere. Hence, \eqref{eq:eigen_local} admits the following sequence of eigenvalues
\begin{equation}\label{eq:eigenvalues_local}
	 \mu_n^1(\lambda):=(n-1)(n+N-3)-\lambda,\quad n\in\N, 
\end{equation}
where we put the upper index $^1$ since we are working in the case $s=1$. From \cite{FFT} (which treats even more general operators with magnetic potentials) we know that the local behavior of solutions to \eqref{eq:local} are governed by eigenvalues and eigefunctions of \eqref{eq:eigen_local}. More precisely, we have the following.
\begin{theorem}[\cite{FFT}, Theorem 1.3]\label{thm:FFT}
	Let $\lambda<(N-2)^2/4$, let $\V\in L^\infty_\loc(B_1\setminus\{0\})$ satisfy \eqref{eq:G_1} and let $u\in H^1(B_1)\setminus\{0\}$ be a non-trivial weak solution of \eqref{eq:local}. Then there exists $n\in\N$ and an eigenfunction $\phi_n^1\in H^1(\mathbb{S}^{N-1})$ of \eqref{eq:eigen_local} such that
	\begin{equation*}
		r^{-\gamma_n^1(\lambda)}u(rx)\to |x|^{\gamma_n^1(\lambda)}\phi_n^1\left(\frac{x}{|x|}\right)\quad\text{as }r\to 0,~\text{in }C^{1,\alpha}_{\loc}(B_1\setminus\{0\}),
	\end{equation*}
	where
	\begin{equation}\label{eq:gamma_n_1}
		\gamma_n^1(\lambda):=-\frac{N-2}{2}+\sqrt{\left(\frac{N-2}{2}\right)^2+\mu_n^1(\lambda)}
	\end{equation}
	and $\mu_n^1(\lambda)$ is as in \eqref{eq:eigenvalues_local}.
	
\end{theorem}
Let us now focus on the case of radial solutions. We denote by
$H^1_\rad(B_1)$ the space of radial functions in $H^1(B_1).$
If $u\in H^1_\rad(B_1)$ is a non-trivial, radial solution of \eqref{eq:local}, it means that the eigenfunction $\phi_n^1$ appearing in \Cref{thm:FFT} must be independent of the angular variable, which means it needs to be constant. However, due to $L^2$ orthogonality, the only constant eigenfunction of \eqref{eq:eigen_local} is the first one, i.e., the one corresponding to   $n=1$.   Indeed, by classical spectral theory, we know that the first eigenvalue   $\mu_1^1(\lambda)=-\lambda$   is simple and has   $\phi_1^1=\mathcal{H}^{N-1}(\mathbb{S}^{N-1})^{-1/2}$    as unique (up to a sign) $L^2$-normalized eigenfunction. Therefore, we have the following corollary, showing that \Cref{thm:main_1} and \Cref{thm:main_2} do not hold if $s=1$.
\begin{corollary}\label{cor:local_radial}
	Let $\lambda<(N-2)^2/4$, let $\V\in L^\infty_{\loc}(B_1\setminus\{0\})$ satisfy \eqref{eq:G_1}	and let $u\in H^1_{\rad}(B_1)\setminus\{0\}$ be a non-trivial, radial weak solution of \eqref{eq:local}
	Then there exists $\kappa\in\R\setminus\{0\}$ such that
	\begin{equation*}
		\lim_{r\to 0}r^{-\gamma_1^1(\lambda)}u(r)=\kappa\quad\text{and}\quad \lim_{r\to 0} r^{-\gamma_1^1(\lambda)+1}u'(r)=\kappa\gamma_1^1(\lambda),
	\end{equation*}
	where
	\begin{equation*}
		\gamma_1^1(\lambda)=-\frac{N-2}{2}+\sqrt{\left(\frac{N-2}{2}\right)^2-\lambda}.
	\end{equation*}
	In particular, if $u\in H^1_{\rad}(B_1)\setminus\{0\}$ is a non-trivial, radial weak solution of
	\begin{equation*}
		-\Delta u=\V u\quad\text{in }B_1,
	\end{equation*}
	then there exists $\kappa\in\R\setminus\{0\}$ such that $\lim_{r\to 0}u(r)=\kappa$.
\end{corollary}

\subsection{\texorpdfstring{The case $0<s<1$}{The case 0<s<1}}\label{sec:nonlocal_behvaior}
In the fractional case $0<s<1$, the scenario is qualitatively different from the one in the case $s=1$. 
What we are going to discuss in this section is mostly contained in \cite{Fall2014}. Here, we fix   $0<s<\min\{1,N/2\}$,   $\lambda<\Lambda_{N,s}$, $\V\in C^1(B_1\setminus\{0\})$ satisfying \eqref{eq:G}, $2<p\leq 2^*_s$ and $a\in\R$ and we focus on weak solutions of
\begin{equation}\label{eq:nonlocal}
	(-\Delta)^s u-\frac{\lambda}{|x|^{2s}}u=\V u+a|u|^{p-2}u \quad\text{in }B_1,
\end{equation}
that is, $u\in \Ds$ and 
\begin{equation*}
	(u,\varphi)_{\Ds}-\lambda\int_{B_1}\frac{u\varphi}{|x|^{2s}}\dx=\int_{B_1}(\V+a|u|^{p-2}) u\varphi\dx\quad\text{for all }\varphi\in C_c^\infty(B_1).
\end{equation*}
We  say that $u$ is a \emph{radial}  solution to \eqref{eq:nonlocal} if we also have $u\in\Dsr$.

Analogously to the case $s=1$, it turns out that the local behavior of weak solutions of \eqref{eq:nonlocal} is governed by a suitable spherical eigenvalue problem. However, a striking difference consists in the fact that such eigenvalue problem has a different nature with respect to the case $s=1$, since it acts on the upper $N$-dimensional sphere
\begin{equation*}
	\mathbb{S}^N_+:=\{(\theta',\theta_{N+1})\in\mathbb{S}^N\colon \theta_{N+1}>0\}.
\end{equation*}
In order to formulate it, we first fix some notation. Let 
\begin{equation*}
	H^1(\mathbb{S}_N^+;\theta_{N+1}^{1-2s}):=\left\{v\in L^1_{\loc}(\mathbb{S}_+^N)\colon \norm{v}_{H^1(\mathbb{S}^N_+;\theta_{N+1}^{1-2s})}^2:= \int_{\mathbb{S}_N^+}\theta_{N+1}^{1-2s}(\abs{\nabla_{\mathbb{S}^N} v}^2+v^2)\,\ds(\theta) <\infty\right\}.
\end{equation*}
It is known that there exists a compact trace operator
\[
H^1(\mathbb{S}^N_+;\theta_{N+1}^{1-2s})\hookrightarrow L^2(\partial \mathbb{S}^N_+)
\]
so that the following characterization holds
\[
H^1_0(\mathbb{S}^N_+;\theta_{N+1}^{1-2s})=\left\{\phi\in H^1(\mathbb{S}^N_+;\theta_{N+1}^{1-2s})\colon \phi=0~\text{in }\partial \mathbb{S}^N_+\right\}.
\]
In an analogous way, if
\[
B_1^+:=\{(x,t)\in\R^{N+1}\colon |x|^2+t^2<1,~t>0\}
\]
we denote
\[
H^1(B_1^+;\,t^{1-2s})=\left\{v\in L^1_{\loc}(B_1^+)\colon \norm{v}_{H^1(B_1^+;t^{1-2s})}^2:= \int_{B_1^+}t^{1-2s}(\abs{\nabla v}^2+v^2)\dx \d t<\infty\right\}.
\]
Again, in view of the compact trace embedding
\[
H^1(B_1^+;t^{1-2s})\hookrightarrow L^2(B_1)
\]
we have the characterization
\[
H^1_{0,B_1}(B_1^+;t^{1-2s})=\left\{ v\in H^1(B_1^+;t^{1-2s})\colon v=0~\text{on }B_1 \right\}.
\]
Now, we consider the following eigenvalue problem:
\begin{equation}\label{spherical}
	\begin{bvp}
		-\dive_{\mathbb{S}^N}(\theta_{N+1}^{1-2s}\,\nabla_{\mathbb{S}^N}\phi)&=\mu\theta_{N+1}^{1-2s}\phi, &&\text{in }\mathbb{S}_+^N, \\
		-\lim_{\theta_{N+1}\to 0^+}\theta_{N+1}^{1-2s}\,\nabla_{\mathbb{S}^N}\phi\cdot \mathbf{e}_{N+1}&=\kappa_s\lambda\phi, &&\text{on }\partial\mathbb{S}_+^N,
	\end{bvp}
\end{equation}
where $\mathbf{e}_{N+1}:=(0,\dots,0,1)\in\mathbb{S}^N_+$ and
\begin{equation*}
	\kappa_s:=\frac{\Gamma(1-s)}{2^{2s-1}\Gamma(s)}.
\end{equation*}
More precisely,  \eqref{spherical} must be intended in a weak sense: we say that $\mu\in\R$ is an \emph{eigenvalue} of and $\phi\in H^1(\mathbb{S}_N^+;\theta_{N+1}^{1-2s})\setminus\{0\}$ is a corresponding \emph{eigenfunction} if
\begin{equation*}
	\int_{\mathbb{S}^N_+}\theta_{N+1}^{1-2s}\nabla_{\mathbb{S}^N}\phi\cdot\nabla_{\mathbb{S}^N}\psi\ds(\theta)-\kappa_s\lambda\int_{\partial \mathbb{S}^N_+}\phi\psi\d\theta'=\mu\int_{\mathbb{S}^N_+}\theta_{N+1}^{1-2s}\phi\psi\ds(\theta)
\end{equation*}
for all $\psi\in H^1(\mathbb{S}_N^+;\theta_{N+1}^{1-2s})$. By classical spectral theory, (being $\theta_{N+1}^{1-2s}$ an $A_2$-Muckenhoupt weight, see \cite[Section 2.1]{Fall2014}) it is known that there exists a discrete sequence of eigenvalues
\[
-\left(\frac{N-2s}{2}\right)^2<\mu_1^s(\lambda)< \mu_2^s(\lambda)\leq \cdots\leq \mu_k^s(\lambda)\leq \cdots
\]
diverging to $+\infty$. Moreover, we have the following characterization
\begin{equation}\label{eq:mu_1}
	\mu_1^s(\lambda)=\min\left\{ \frac{\displaystyle \int_{\mathbb{S}^N_+}\theta_{N+1}^{1-2s}|\nabla_{\mathbb{S}^N}\phi|^2\ds(\theta)-\kappa_s\lambda\int_{\partial \mathbb{S}^N_+}\phi^2\d\theta' }{\displaystyle \int_{\mathbb{S}^N_+}\theta_{N+1}^{1-2s}\phi^2\ds(\theta)}\colon\phi\in H^1(\mathbb{S}^N_+;\theta_{N+1}^{1-2s})\setminus\{0\} \right\}.
\end{equation}
In \cite{Fall2014}, the authors prove that the local asymptotic behavior of solutions of \eqref{eq:nonlocal} is governed by eigenvalues and eigenfunctions of \eqref{spherical}. More precisely, we have the following.
\begin{theorem}[\cite{Fall2014}, Theorem 1.1]\label{thm:FF}
	Let $\lambda<\Lambda_{N,s}$, $\V\in C^1(B_1\setminus\{0\})$ satisfying \eqref{eq:G}, $2<p\leq 2^*_s$ and $a\in\R$ and let $u\in\Ds$ be a non-trivial weak solution of \eqref{eq:nonlocal}. Then $u\in C^{1,\alpha}_\loc(B_1\setminus\{0\})$ (for some $\alpha\in(0,1)$) and there exists $n\in\N$ and an eigenfunction $\phi_n^s\in H^1(\mathbb{S}^N_+;\theta_{N+1}^{1-2s})\setminus\{0\}$ of \eqref{spherical} such that
	\begin{equation*}
		r^{-\gamma_n^s(\lambda)}u(rx)\to |x|^{\gamma_n^s(\lambda)}\phi_n^s\left(\frac{x}{|x|},0\right)\quad\text{as }r\to 0,~\text{in }C^{1,\alpha}_\loc(B_1\setminus\{0\}),
	\end{equation*}
	where
	\begin{equation}\label{eq:gamma_n_s}
		\gamma_n^s(\lambda):=-\frac{N-2s}{2}+\sqrt{\left(\frac{N-2s}{2}\right)^2+\mu_n^s(\lambda)}
	\end{equation}
	and $\mu_n^s(\lambda)$ is the eigenvalue of \eqref{spherical} corresponding to $\phi_n^s$. Moreover, if $\lambda=0$, then $u\in C^{1,\alpha}_\loc(B_1)$, for some $\alpha\in(0,1)$.
\end{theorem}

As a corollary, we are able to derive the asymptotic local behavior near the origin of radial solutions of \eqref{eq:nonlocal}. Before stating it, we need the following boundary unique continuation result for eigenfunctions.

\begin{lemma}\label{lemma:UC}
	Let $\mu\in\R$ be an eigenvalue of \eqref{spherical} and let $\phi\in H^1(\mathbb{S}^N_+;\theta_{N+1}^{1-2s})\setminus\{0\}$ be a corresponding eigenfunction. Then $\phi\not\equiv 0$ in $\partial \mathbb{S}^N_+$.
\end{lemma}
\begin{proof}
	Let us assume by contradiction that $\phi\equiv 0$ on $\partial \mathbb{S}^N_+$. We let 
	\[
	u(r,\theta):=r^\gamma\phi(\theta),
	\]
	where
	\[
	\gamma:=-\frac{N-2s}{2}+\sqrt{\left(\frac{N-2s}{2}\right)^2+\mu}
	\]
	is such that
	\[
	\gamma(\gamma+N-2s)=\mu.
	\]
	One can easily check that $u\in H^1(B_1^+;t^{1-2s})$ and that, since $\phi=0$ on $\partial\mathbb{S}_+^N$, then $u=0$ on $B_1$. Moreover, by the choice of $\gamma$ and direct computations in polar coordinates, one can easily prove that
	\begin{equation}\label{eq:1}
		\int_{B_1^+}t^{1-2s}\nabla u\cdot\nabla\phi\dxdt=0\quad\text{for all }\phi\in C_c^\infty(B_1^+\cup B_1).
	\end{equation}
	Let us now set, for $(x,t)\in B_1^{N+1}:=\{z\in\R^{N+1}\colon |z|<1\}$
	\[
	\tilde{u}(x,t):=\begin{cases}
		u(x,t)&\text{if }t>0, \\
		0&\text{if }t\leq 0.
	\end{cases}
	\]
	One can easily see that, since $u=0$ on $B_1$, then $\tilde{u}\in H^1(B_1^{N+1};\abs{t}^{1-2s})$. Moreover, for $\varphi\in C_c^\infty(B_1^{N+1})$, in view of \eqref{eq:1} we know that
	\[
	\int_{B_1^{N+1}}\abs{t}^{1-2s}\nabla \tilde u\cdot\nabla \varphi\dx \d t=\int_{B_1^+}t^{1-2s}\nabla u\cdot\nabla \varphi\dx \d t=0,
	\]
	that is
	\[
	\dive(\abs{t}^{1-2s}\nabla \tilde{u})=0\quad\text{in }B_1^{N+1}.
	\]
	From the strong unique continuation property for elliptic equations with $A_2$-Muckenhoupt weights (such as $\abs{t}^{1-2s}$), see e.g. \cite[Corollary 1.4]{TZ}, we know that, since $\tilde{u}$ vanishes on an open set, then it must be identically $0$ over the whole $B_1^{N+1}$, thus meaning that $\phi\equiv 0$ on $\mathbb{S}^N_+$, which is a contradiction. This concludes the proof.
\end{proof}

We are now ready to provide the local behavior of radial weak solutions of \eqref{eq:nonlocal} near the origin. Essentially, this comes from the fact that eigenfunctions of \eqref{spherical} depending only on $\theta_{N+1}$ are non-zero constants on $\partial\mathbb{S}^N_+$.

\begin{corollary}\label{cor:nonlocal_radial}
	Let $\lambda<\Lambda_{N,s}$, $\V\in C^1(B_1\setminus\{0\})$ satisfying \eqref{eq:G}, $2<p\leq 2^*_s$ and $a\in\R$ and let $u\in\Dsr$ be a non-trivial, radial weak solution of \eqref{eq:nonlocal}. Then there exists $\kappa\in\R\setminus\{0\}$ and $n\in\N$ such that, if $\gamma_n^s(\lambda)$ is as in \eqref{eq:gamma_n_s}, the following holds
	\begin{equation*}
		\lim_{r\to 0}r^{-\gamma_n^s(\lambda)}u(r)=\kappa\quad\text{and}\quad \lim_{r\to 0}r^{-\gamma_n^s(\lambda)+1}u'(r)=\kappa\gamma_n^s(\lambda).
	\end{equation*}
\end{corollary}
\begin{proof}
	From \Cref{thm:FF} we know that there exists $n\in\N$ and an eigenfunction $\phi_n^s\in H^1(\mathbb{S}^N_+;\theta_{N+1}^{1-2s})\setminus\{0\}$ of \eqref{spherical} such that
	\begin{equation*}
		r^{-\gamma_n^s(\lambda)}u(rx)\to |x|^{\gamma_n^s(\lambda)}\phi_n^s\left(\frac{x}{|x|},0\right)\quad\text{as }r\to 0,~\text{in }C^{1,\alpha}_\loc(B_1\setminus\{0\}),
	\end{equation*}
	where $\gamma_n^s(\lambda)$ is as in \eqref{eq:gamma_n_s}. Now, since $u$ is radial, then $\phi_n^s(\theta',0)= \kappa$ for some $\kappa\neq 0$ and for all $\theta'\in\partial\mathbb{S}^N_+$, and this concludes the proof.	
\end{proof}
We conclude the section with a characterization of $\gamma_\lambda$ from \eqref{eq:gamma_lamda},   i.e., of $\gamma_1^s(\lambda)$.  
\begin{lemma}\label{lemma:Psi_gamma}
	For any $\gamma\in(-2s,N-2s)$, we let
	\begin{equation*}
		\Psi_{N,s}(\gamma):=2^{2s}\frac{\Gamma(\tfrac{N-\gamma}{2})\,\Gamma(\tfrac{2s+\gamma}{2})}{\Gamma(\frac{N-2s-\gamma}{2})\,\Gamma(\frac\gamma2)}
	\end{equation*}
	Then
	\begin{equation*}
		\Psi_{N,s}\colon \left(-2s,\frac{N-2s}{2}\right]\to \left(-\infty,\Lambda_{N,s}\right]
	\end{equation*}
	is a strictly increasing bijection and
	\begin{equation}\label{eq:Psi_gamma_th1}
		\gamma_\lambda=-\Psi_{N,s}^{-1}(\lambda)\quad\text{for all }\lambda\in\left(-\infty,\Lambda_{N,s}\right]
	\end{equation}
	where $\gamma_\lambda$ is as in \eqref{eq:gamma_lamda}.
\end{lemma}
\begin{proof}
	The fact that $\Psi_{N,s}$ is a strictly increasing bijection is proved in \cite[Lemma 2.3]{ChenWeth}, so we pass to the proof of \eqref{eq:Psi_gamma_th1}. Let $-2s<\gamma<(N-2s)/2$ and let $u_\gamma(x):=|x|^{-\gamma}$. In view of \cite[Lemma 2.1]{ChenWeth}, we know that $u_\gamma\in L^2(\R^N;(1+|x|^{N+2s})^{-1}\dx)$, which implies $(-\Delta)^su_\gamma\in \mathcal{S}'(\R^N)$, and that 
	\begin{equation*}
		(-\Delta)^s u_\gamma=\Psi_{N,s}(\gamma)|x|^{-2s}u_\gamma\quad\text{in }\mathcal{S}'(\R^N).
	\end{equation*}
	Moreover, it is a standard procedure to extend the function $u_\gamma\colon\R^N\to\R$ to a function
	\begin{equation*}
		U_\gamma\colon\R^{N+1}_+\to\R,
	\end{equation*}
	such that $U_\gamma(x,0)=u_\gamma(x)$ for all $x\neq 0$ and satisfying
	\begin{equation}\label{eq:gamma_ext}
		\begin{bvp}
			-\dive(t^{1-2s}\nabla U_\gamma)&=0,&&\text{in }\R^{N+1}_+, \\
			-\lim_{t\to 0^+}t^{1-2s}\frac{\partial U_\gamma}{\partial t}&=\kappa_s(-\Delta)^s u_\gamma, &&\text{on }\R^N,
		\end{bvp}
	\end{equation}
	in the sense that
	\begin{equation*}
		-\int_{\R^{N+1}_+}u\dive(t^{1-2s}\nabla \varphi)\dxdt=\kappa_s\int_{\R^N}u(-\Delta)^s\varphi\dx\quad\text{for all }\varphi\in \mathcal{S}(\overline{\R^{N+1}_+}).
	\end{equation*}
	This procedure is done via convolution with the Poisson kernel
	\[
		P_t(x):=p_{N,s}\frac{t^{2s}}{\left(|x|^2+t^2\right)^{\frac{N+2s}{2}}},\quad\text{with }p_{N,s}:=\pi^{-\frac{N}{2}}\frac{\Gamma\left(\frac{N+2s}{2}\right)}{\Gamma(s)}
	\]
	see \cite{CaffSilv} and \cite{CS}, i.e.
	\begin{equation*}
		U_\gamma(x,t):=(P_t\ast u_\gamma)(x)\quad\text{for }x\in\R^N,~t>0.
	\end{equation*}
	Since $\gamma<(N-2s)/2<N-2s$, one can easily see that 
	\begin{equation*}
		|(-\Delta)^{s/2}u_\gamma|\in L^2_\loc(\R^N), 
	\end{equation*}
	which readily implies that $U_\gamma\in H^1(B_1^+;t^{1-2s})$.	As a consequence, we have that \eqref{eq:gamma_ext} holds in a weak sense in $B_1^+$, that is
	\begin{equation*}
		\int_{B_1^+}t^{1-2s}\nabla U_\gamma\cdot\nabla\varphi\dxdt=\kappa_s\Psi_{N,s}(\gamma)\int_{B_1}\frac{U_\gamma\varphi}{|x|^{2s}}\dx\quad\text{for all }\varphi\in C^\infty_c(B_1^+\cup B_1).
	\end{equation*}
	We now let $\gamma:=\Psi_{N,s}^{-1}(\lambda)$, for $\lambda <\Lambda_{N,s}$, and we consider $V_\lambda:=U_{\Psi_{N,s}^{-1}(\lambda)}$, which weakly solves
	\begin{equation*}
		\begin{bvp}
			-\dive(t^{1-2s}\nabla V_\lambda)&=0, &&\text{in }B_1^+, \\
			-t^{1-2s}\frac{\partial V_\lambda}{\partial t}-\kappa_s\lambda\frac{V_\lambda}{|x|^{2s}}&=0, &&\text{on }B_1.
		\end{bvp}
	\end{equation*}
	Now, in view of \cite[Theorem 4.1]{Fall2014}, we know that there exists an eigenvalue $\mu_n^s(\lambda)$ of \eqref{spherical} and a corresponding eigenfunction $\phi_n^s$ such that
	\begin{equation*}
		r^{-\gamma_n^s(\lambda)}V_\lambda(rx,0)\to |x|^{\gamma_n^s(\lambda)}\phi_n^s\left(\frac{x}{|x|},0\right)\quad\text{as }r\to 0,~\text{in }C^{1,\alpha}_\loc(B_1\setminus\{0\}),
	\end{equation*}
	up to restricting $\alpha\in(0,1)$, where $\gamma_n^s(\lambda)$ is as in \eqref{eq:gamma_n_s}. But since $V_\lambda(x,0)=|x|^{-\Psi_{N,s}^{-1}(\lambda)}>0$ in $B_1\setminus\{0\}$, this forces $n=1$ and $\gamma_\lambda=-\Psi_{N,s}^{-1}(\lambda)$. By passing to the limit, the same holds for the extremal value  $\lambda=\Lambda_{N,s}$, thus concluding the proof.
\end{proof}

\subsection{Remarks and examples}

In the previous sections, we recalled what is the local behavior, near the origin, of solutions of Schr\"odinger equations with Hardy potential: it turns out that, both in the fractional \eqref{eq:nonlocal} and non-fractional \eqref{eq:local} case, such behavior can be expressed in terms of eigenvalues and eigenfunctions of a suitable corresponding eigenvalue problem, i.e. \eqref{spherical} and \eqref{eq:eigen_local} respectively. In particular, the eigenvalue is related with the vanishing order of the solution (through formulas \eqref{eq:gamma_n_s} and \eqref{eq:gamma_n_1}), while a corresponding eigenfuction determines the {shape} of the solution near the origin (see \Cref{thm:FF} and \Cref{thm:FFT}). We also emphasized  that, while in the local case the spherical eigenvalue problem lives in $\mathbb{S}^{N-1}$, in the fractional case the corresponding eigenvalue problem lives in the upper half-sphere in one more dimension, i.e. $\mathbb{S}^N_+$. This comes as a consequence of the fact that the non-local operator $(-\Delta)^s$ in $\R^N$ can be seen as a local boundary Dirichlet-to-Neumann operator in $\R^{N+1}_+$,   as shown by Caffarelli and Silvestre \cite{CaffSilv}.  

When dealing with radial solutions, a remarkable difference emerges. Indeed, while in the local case the eigenfunction determining the shape (near the origin) of a non-trivial radial solution is forced to be the first one, which is a non-zero constant, this does not happen in the fractional case. In particular, when $0<s<1$, the behavior near the origin of a non-trivial radial solution is shaped by an eigenfunction on $\mathbb{S}^N_+$ which only needs to be a non-zero constant at the boundary $\partial\mathbb{S}^N_+$, without necessarily being the first one, i.e. constant on the whole $\mathbb{S}^N_+$. In particular, when $\lambda=0$ this reflects in the possible existence of non-trivial radial solutions which vanish at the origin. 

In order to better clarify this point, let us focus on a particular example. Set $N=2$, $s=1/2$, so that $2s<N<6s$ and $\lambda=0$. In view of \Cref{thm:FF} and \Cref{cor:nonlocal_radial}, if we consider a non-trivial, radial $u\in\mathcal{D}^{\frac{1}{2},2}_\rad(\R^2)$ weakly solving 
\begin{equation*}
	(-\Delta)^{\frac{1}{2}}u=\V u\quad\text{in }B_1
\end{equation*}
with $\V\in C^1(B_1\setminus\{0\})$ satisfying \eqref{eq:G}, then the local behavior of $u$ near the origin is governed by the eigenelements of the following eigenvalue problem
\begin{equation}\label{eq:spherical_harm}
	\begin{bvp}
		-\Delta\phi&=\mu\phi,&&\text{in }\mathbb{S}^2_+, \\
		\frac{\partial\phi}{\partial \theta_3}&=0,&&\text{on }\partial\mathbb{S}^2_+,
	\end{bvp}
\end{equation}
which is \eqref{spherical} with $s=1/2$. By an even reflection, this is equivalent to the eigenvalue problem for the spherical Laplacian in $\mathbb{S}^2$ with even eigenfunctions (hence, non-identically vanishing on $\partial\mathbb{S}^2_+$, see \Cref{lemma:UC}). In particular, such eigenfunctions form a subset of the so called \emph{spherical harmonics} which is nontrivial since, for instance, the restriction of 
\begin{equation}\label{eq:P}
	P(x_1,x_2,x_3):=35x_3^2-30x_3^2(x_1^2+x_2^2+x_3^2)+3(x_1^2+x_2^2+x_3^2)^2
\end{equation}
to $\mathbb{S}^2_+$ belongs to it, with corresponding eigenvalue $\mu=20$.

From the analysis we made in \Cref{sec:nonlocal_behvaior}, we know that there exists an eigenvalue $\mu_n\geq 0$ of \eqref{eq:spherical_harm} and a corresponding eigenfunction $\phi_n$ such that
$\phi_n (\theta',0)=\kappa\neq 0$ for all $\theta'\in\partial\mathbb{S}^2_+=\mathbb{S}^1$ and
\begin{equation*}
	r^{-\gamma_n} u(rx)\to |x|^{\gamma_n}\phi_n\left(\frac{x}{|x|},0\right)=\kappa |x|^{\gamma_n}\quad\text{as }r\to 0,~\text{in }C^{1,\alpha}_\loc(B_1),
\end{equation*}
where
\begin{equation*}
	\gamma_n:=-\frac{1}{2}+\sqrt{\frac{1}{4}+\mu_n}.
\end{equation*}
Besides the case $n=1$, for which $\mu_n^{\frac{1}{2}}(0)=\mu_1^{\frac{1}{2}}(0)=0$ and $\gamma_n^{\frac{1}{2}}(0)=\gamma_1^{\frac{1}{2}}(0)=0$, so that the solution $u$ does not vanish at $0$,  it might happen that $n\geq 2$ so that the solution is vanishing at $0$, still being non-trivial (for instance, $P$ as in \eqref{eq:P} corresponds to the case $n=4$, so that $\mu_4^{\frac{1}{2}}(0)=20$). The rest of the paper is devoted to exhibit examples of such phenomenon.

\section{Non-uniqueness}\label{sec:main}

The aim of this section is to prove our main results \Cref{thm:main_1} and \Cref{thm:main_2}. Before going on, we need to recall a crucial result, which is peculiar of the fractional Laplacian (we refer to \cite{GSU} for the proof).
\begin{theorem}[\cite{GSU}, Theorem 1.2]\label{thm:UCP}
	Let $\Omega\sub\R^N$ be open and non-empty and let    $u\in H^s(\mathbb R^N)$   be such that $u\equiv (-\Delta)^s u\equiv 0$ in $\Omega$. Then $u\equiv 0$ in $\R^N$.
\end{theorem}

We are now ready to prove our first main result.
\begin{proof}[Proof of \Cref{thm:main_1}]
	For the sake of simplicity, in the following we simply write $\gamma$ in place of $\gamma_\lambda$,   where $\gamma_\lambda=\gamma_1^s(\lambda)$ is given by \eqref{eq:gamma_n_s}, see also Lemma \ref{lemma:Psi_gamma}.   Let $2<p<2^*_s$ be such that
	\begin{equation*}
		\frac{2^*_s}{p-2}\geq q
	\end{equation*}
	 and let $u_1,u_2\in H^s_{0,\rad}(B_1)$ be two distinct nontrivial solutions of \eqref{nonlinearp}, whose existence is ensured by \Cref{multiplicity}. In view of \Cref{thm:FF} and \Cref{cor:nonlocal_radial} we have that $u_1,u_2\in C^{1,\alpha}_{\loc}(B_1\setminus\{0\})$ for some $\alpha\in(0,1)$  and there exists $\gamma_1,\gamma_2 \in [\gamma,+\infty)$ and $\kappa_1,\kappa_2\in\R\setminus\{0\}$ such that
	\begin{equation}\label{eq:main_1_1}
		\lim_{r\to 0}r^{-\gamma_i}u_i(r)=\kappa_i\quad\text{and}\quad \lim_{r\to 0}r^{-\gamma_i+1}u_i'(r)=\kappa_i\gamma_i\quad\text{for }i=1,2.
	\end{equation}
	Up to a change of sign, we can assume $\kappa_i>0$ and, without loss of generality, that $\kappa_1\geq \kappa_2$. From \eqref{eq:main_1_1} we deduce that, for any $i=1,2$, there exists $c_i\in\{0,\kappa_i\}$ and   $c_i'\in\{0,\kappa_i\gamma\}$   such that
	\begin{equation}\label{eq:main_1_2}
		\lim_{r\to 0}r^{-\gamma}u_i(r)=c_i\quad\text{and}\quad\lim_{r\to 0}r^{-\gamma+1}u_i'(r)=c_i'.
	\end{equation}
	In particular, $c_i=0$ if and only if $\gamma_i>\gamma$. 
	
		  Let us first show how to conclude the proof if $c_i=0$ for some $i\in\{1,2\}$, i.e., if $\gamma_i>\gamma$.
	We let in this case $\V_i:=|u_i|^{p-2}$ and we claim  that  $\V_i\in C^{1,\alpha}_{\loc}(B_{\bar r}\setminus\{0\})\cap L^q(B_{\bar r})$, up to choosing $\bar r\in (0,1]$ in such a way that $u_i$ is positive on $B_{\bar r}\setminus\{0\}$ and thus bounded away from $0$ on any compact subset of $B_{\bar r}\setminus\{0\}$ (the existence of such $\bar r$ is ensured by \eqref{eq:main_1_1}, since $\kappa_i>0$). The claim easily follows by definition of $\V_i$, Sobolev inequality and the fact that $u_i\in C^{1,\alpha}_{\loc}(B_1\setminus\{0\})$ and that the map $t\mapsto |t|^{p-2}$ is smooth outside $t=0$. We also prove that $\V_i$ satisfies the condition \eqref{eq:G}:  
	thanks to \eqref{eq:main_1_2} we derive that
	\begin{equation*}
		|\V_i(x)|+|\nabla \V_i(x)\cdot x|=|u_i(r)|^{p-2}+(p-2)|u_i(r)|^{p-3}|u_i'(r)|r\leq C r^{\gamma(p-2)},\quad\text{as }r\to 0,
	\end{equation*}
	for some $C>0$ independent of $x$, where $r=|x|$. In addition, since 
	\begin{equation}\label{eq:main_1_5}
		\gamma>-\frac{N-2s}{2}\quad\text{and}\quad 2<p<2^*_s
	\end{equation}
	we have that $\gamma(p-2)>-2s$, thus proving that $\V_i$ satisfies \eqref{eq:G}.
	  By taking into account \eqref{nonlinearp}, \eqref{eq:main_1_2} and the definition of $\V_i$, since we are assuming $c_i=0$, we see that $u_i$ satisfies \eqref{limitcondition} and
	\begin{equation*}
			(-\Delta)^su_i-\frac{\lambda}{|x|^{2s}} u_i =\V_i u_i \quad\text{in }B_{\bar r},
\end{equation*}
	thus the proof concludes by suitably scaling. Indeed,  by choosing  $u(x):=u_i(\bar r x)$, $\V(x):=\bar r^{2s}\V_i(\bar r x)$ and $R:=\bar r^{-1}$, we have that $u$ satisfies \eqref{eq:main} and \eqref{limitcondition} and that $\V$ belongs to $ C^{1,\alpha}_{\loc}(B_1\setminus\{0\}) \cap L^q(B_1)$, satisfies \eqref{eq:G} and $\V>0$ in $B_1\setminus\{0\}$. Notice that $u\in H^s(\mathbb R^N)$, $u\equiv0$ outside $B_R$ since $u_i\equiv0$ outside $B_1$. Moreover, if $\Omega$ is a  nonempty open set in $B_R$ then $u$ cannot vanish identically in $\Omega$: indeed, if $u\equiv 0$ in an open set $\Omega\sub B_R$, then $u_i\equiv 0$ in an open subset of $B_1$. At this point, since $u_i$ weakly solves \eqref{nonlinearp}, by \Cref{thm:UCP} we would conclude that $u_i\equiv 0$ in $\R^N$, and this is a contradiction.
	
	In order to proceed, we hereafter assume that $c_i=\kappa_i>0$ and that $\gamma_i=\gamma$ for both $i=1,2$. We now let
	\begin{equation*}
		u_{2,\delta}(x):=\delta\,u_2\left(\delta^{\frac{p-2}{2s}}\,x\right)\quad\text{where}\quad \delta:=\left(\frac{\kappa_1}{\kappa_2}\right)^{\left(1+\frac{\gamma(p-2)}{2s}\right)^{-1}},	
	\end{equation*}
	so that $\delta\geq 1$ and there holds
	\begin{equation}\label{eq:main_1_3}
		\lim_{r\to 0}r^{-\gamma}u_{2,\delta}(r)=\lim_{r\to 0}r^{-\gamma}u_1(r)=\kappa_1>0
	\end{equation}
	and
	\begin{equation}\label{eq:main_1_4}
		\lim_{r\to 0}r^{-\gamma+1}u_{2,\delta}'(r)=\lim_{r\to 0}r^{-\gamma+1}u_1'(r)=\kappa_1\gamma.
	\end{equation}
	Moreover, if we denote $R_{\delta}:=\delta^{-\frac{p-2}{2s}}\leq 1$, one can easily check that $u_{2,\delta}\in H^s_{0,\rad}(B_{R_{\delta}})$ weakly solves
	\begin{equation*}
		\begin{bvp}
			(-\Delta)^s u-\frac{\lambda}{|x|^{2s}}u&=|u|^{p-2}u, &&\text{in }B_{R_{\delta}}, \\
			u&=0, &&\text{in }\R^N\setminus B_{R_{\delta}}.
		\end{bvp}
	\end{equation*}
	At this point, we define $w:=u_1-u_{2,\delta}$,   so  that $w\in H^s_{0,\rad}(B_1)$.   First of all, we notice that, by construction
	\[
		\lim_{r\to 0}r^{-\gamma}w(r)=0.
	\]
	Moreover, if we let $g(t):=|t|^{p-2}t$ and
	\begin{equation}\label{eq:def_h}
		h(t_1,t_2):=\begin{cases}\displaystyle 
			\frac{g(t_1)-g(t_2)}{t_1-t_2}, &\text{for }t_1\neq t_2, \\
			g'(t_1), &\text{for }t_1=t_2,
		\end{cases}
	\end{equation}
	we have that $w$  weakly solves
	\begin{equation}\label{equazionew}
			(-\Delta )^s w-\frac{\lambda}{|x|^{2s}}w=\V_0 w\quad\text{in }B_{R_\delta}, 
	\end{equation}
	where $\V_0:=h(u_1,u_{2,\delta})$. We now claim that, for some $\rho\in(0,R_{\delta}]$, $\V_0\in  C^{1,\alpha}_{\loc}(B_\rho\setminus\{0\})\cap L^q(B_\rho)$, $\V>0$ in $B_\rho\setminus\{0\}$ and that $\V_0$ satisfies \eqref{eq:G}. More precisely, $\rho\in(0,R_\delta]$ is chosen in such a way that $u_1,u_{2,\delta}>0$ in $B_\rho\setminus\{0\}$ (the existence of such $\rho$ is ensured by \eqref{eq:main_1_3}). As a first step, we observe that, since $g\in C^\infty(\R\setminus\{0\})$, then $h\in C^\infty((0,+\infty)\times(0,+\infty))$ and, as a consequence, since $u_1,u_{2,\delta}\in C^{1,\alpha}_{\loc}(B_\rho\setminus\{0\})$ and $u_1,u_{2,\delta}>0$ in $B_\rho\setminus\{0\}$, then $\V_0\in C^{1,\alpha}_{\loc}(B_\rho\setminus\{0\})$ as well. Moreover, since $g$ is increasing, we have $\V_0>0$ in $B_\rho\setminus\{0\}$. We now prove that $\V_0\in L^q(B_\rho)$. In view of Lagrange theorem, for any $t_1,t_2>0$ there exists $t_{12}>0$ such that
	\begin{equation*}
		0<t_{12}\leq \max\{t_1,t_2\}\quad\text{and}\quad h(t_1,t_2)=g'(t_{12})=(p-1)t_{12}^{p-2},
	\end{equation*}
	which implies that
	\begin{equation*}
		|h(t_1,t_2)|\leq C(|t_1|^{p-2}+|t_2|^{p-2})\quad\text{for all }t_1,t_2>0,
	\end{equation*}
	for some $C>0$ independent of $t_1,t_2$. Therefore, since $u_i\in L^{2^*_s}(B_\rho)$, we have that $\V_0\in L^{2^*_s/(p-2)}(B_\rho)\sub L^q(B_\rho)$.	Let us now check the growth condition at the origin \eqref{eq:G}. By homogeneity, still denoting by $\V_0$ the radial profile of $\V_0$, we have that
	\begin{equation*}
		\V_0(r)=r^{\gamma(p-2)}h(r^{-\gamma}u_1(r),r^{-\gamma}u_{2,\delta}(r))
	\end{equation*}
	and using Euler's theorem for positive homogeneous functions 
	\begin{equation*}
		\V_0'(r)=r^{\gamma(p-2)-1}\nabla h(r^{-\gamma}u_1(r),r^{-\gamma}u_{2,\delta}(r))\cdot(r^{-\gamma+1}u_1'(r),r^{-\gamma+1}u_{2,\delta}'(r)),
	\end{equation*}
	where $r=|x|\in(0,\rho)$. Now, thanks to the regularity of $h$ and in view of \eqref{eq:main_1_3} and \eqref{eq:main_1_4}, we have that
	\begin{equation*}
		\lim_{r\to 0}r^{-\gamma(p-2)}\V_0(r)=h(\kappa_1,\kappa_1)=(p-1)\kappa_1^{p-2}
	\end{equation*}
	and that
	\begin{equation*}
		\lim_{r\to 0}r^{-\gamma(p-2)+1}\V_0'(r)=\nabla h(k_1,k_1)\cdot(\kappa_1\gamma,\kappa_1\gamma)=(p-1)(p-2)\kappa_1^{p-2}\gamma.
	\end{equation*}
	Moreover, in view of \eqref{eq:main_1_5} we deduce that $\gamma(p-2)>-2s$, which completes the proof of the claim. 
	  We conclude by crucially exploiting Theorem \ref{thm:UCP}, which allows to prove  that $w$ is nontrivial.  
	In particular, we claim that for all open, nonempty $\Omega\sub B_1$, we have $w\not\equiv 0$ in $\Omega$. Indeed,   suppose by contradiction that $w\equiv 0$ in some nonempty open set $\Omega\sub B_1$. If $\delta<1$ and  $\Omega \cap  (B_1\setminus \overline{B_{R_\delta}})\neq \emptyset$, then $u_1\equiv w\equiv 0$ in the open set $\Omega\cap (B_1\setminus \overline{B_{R_\delta}})$ and, in view of the equation satisfied by $u_1$ in $B_1$ and of \Cref{thm:UCP}, this would imply that $u_1\equiv 0$ in $\R^N$, which is a contradiction being $u_1$ nontrivial. On the other hand, if $\delta<1$ and $\Omega\sub \overline{B_{R_\delta}}$, thanks to   \eqref{equazionew}  and \Cref{thm:UCP}, we would obtain that $w\equiv 0$ in $\R^N$, hence $u_1\equiv 0$ in $B_1\setminus \overline{B_{R_\delta}}$, thus reaching a contradiction as in the previous case. Else if $\delta=1$, i.e., $\kappa_1=\kappa_2$, then \eqref{equazionew} and Theorem \ref{thm:UCP} imply $u_1\equiv u_2$ in $\R^N$, thus contradicting the fact that $u_1,u_2$ are distinct $H^s_{0,\rad}(B_1)$ solutions to \eqref{nonlinearp}. The claim is proven.  
	 Finally, in order to recover a solution on the ball $B_1$, we just scale and define 	 
	\begin{equation*}
		u(x):=w(\rho x),\quad\V(x):=\rho^{2s}\V_0(\rho x),\quad R:=\rho^{-1}.
	\end{equation*}
	In this way, since $\V_0\in C^{1,\alpha}_{\loc}(B_\rho\setminus\{0\})\cap L^q(B_\rho)$, $\V_0>0$ in $B_\rho\setminus\{0\}$ and since $\V_0$ satisfies \eqref{eq:G}, we get   $\V\in C^{1,\alpha}_{\loc}(B_1\setminus\{0\})\cap L^q(B_1)$, $\V>0$ in $B_1\setminus\{0\}$,  $\V$ satisfies \eqref{eq:G} as well, and $u$ satisfies \eqref{limitcondition} and \eqref{eq:main}. Moreover, $u\equiv 0$ outside $B_R$ since $w\equiv 0$ outside $B_1$,  and having shown that $w$ cannot vanish identically on a nonempty open subset of $B_1$, we obtain that $u$ cannot vanish identically on a nonempty open subset of $B_R$.
   
\end{proof}

We now prove our second main result.
\begin{proof}[Proof of \Cref{thm:main_2}]
	This proof closely follows the one of \Cref{thm:main_1}, hence we only sketch it.   In the case $\lambda=0$, we have $\gamma_1^s(0)=0$, where $\gamma_1^s(0)$ is given by \eqref{eq:gamma_n_s}, see also Lemma \ref{lemma:Psi_gamma}.  Let $0<\alpha<\tfrac{4s}{N-2s}$.   If $N\geq 6s$, let $p\in(\alpha+2,2^*_s)$ while, if $N<6s$, let $p\in(3,2^*_s)$. 
	As in the proof of \Cref{thm:main_1}, let $u_1,u_2\in H^s_{0,\rad}(B_1)$ be two distinct, nontrivial solutions of \eqref{nonlinearp}, whose existence is ensured by \Cref{multiplicity}. In view of \Cref{thm:FF}  and \Cref{cor:nonlocal_radial} we have that, for $i=1,2$, $u_i\in C^1(B_1)$ and    there exist $\gamma_1,\gamma_2\in[0,+\infty)$ such that \eqref{eq:main_1_1}   holds for some $\kappa_i\neq 0$ (it is not restrictive to assume $\kappa_i>0$),  so that there exist $c_i\in\{0,\kappa_i\}$ such that $u_i(0)=c_i$, with $c_i=0$ if and only if $\gamma_i>0$. Radiality implies $u'_i(0)=0$. 
	
	Let us consider the case in which  $u_i(0)=0$ (for either $i=1$ or $i=2$).
	By \eqref{eq:main_1_1} we can find $\bar r\in(0,1]$ such that $u_i>0$ in $B_{\bar r}\setminus\{0\}$.  
	Letting $\V_i:=|u_i|^{p-2}$, one can easily check that  $\V_i\in C^1(B_{\bar r}\setminus\{0\})$, $\V_i>0$ in $B_{\bar{r}}\setminus\{0\}$, that $\V_i\in C^{0,\alpha}_{\loc}(B_1)$ if $N\geq 6s$  (since $p-2>\alpha$), that $\V_i\in C^1(B_1)$ if $N<6s$ and that $\V_i$ satisfies \eqref{eq:G}.  One concludes   by reasoning as in the proof of Theorem \ref{thm:main_1},  choosing $u(x):=u_i(\bar r x)$, $\V(x)=\bar r^{2s}\V_i(\bar r x)$ and $R:=\bar r^{-1}$. In this way, one has indeed $\V \in C^{1}(B_1\setminus\{0\})$, $\V>0$ in $B_1\setminus\{0\}$, $\V\in C^{0,\alpha}_{\loc}(B_1)$ if $N\ge 6s$, $\V\in C^1(B_1)$ if $N<6s$, $u(0)=0$, $(-\Delta)^s u=\V u$ in $B_1$,
	$u\equiv 0$ outside $B_{R}$ and $u\not\equiv0$ on any open set contained in $B_{R}$ . 
	
	Let us  consider the case $u_i(0)=\kappa_i\neq 0$ (it is not restrictive to assume $\kappa_1\geq \kappa_2>0$). Let
	\[
		u_{2,\delta}(x):=\delta \, u_2(\delta^{\frac{p-2}{2s}}\,x)\quad\text{with}\quad \delta:=\frac{\kappa_1}{\kappa_2}\ge 1,
	\]
	  so that \eqref{eq:main_1_3} holds with $\gamma=0$,  
	and finally let $w:=u_1-u_{2,\delta}$. Moreover, let $R_\delta:=\delta^{-\frac{p-2}{2s}}\leq 1$ and
	$
		\V_0:=h(u_1,u_{2,\delta}),
	$
	where $h$ is defined in \eqref{eq:def_h}. Since $h\in C^{0,p-2}_{\loc}(\R^2)$ if $p\leq 3$ (corresponding to the case $N\geq 6s$) and $h\in C^1(\R^2)$ if $p>3$ (corresponding to the case $N<6s$),   and since $u_1,u_{2,\delta}\in C^1(B_{R_{\delta}})$, we notice that $\V_0\in C^{0,\alpha}_{\loc}(B_{R_{\delta}})$  if $N\geq 6s$ and $\V_0\in C^1(B_{R_{\delta}})$   if $N<6s$.
 	We also have $\V_0\in C^1(B_\rho)$ for some suitable $\rho\in (0,R_\delta]$, even if $N\ge 6s$, since a small enough $\rho$ can be chosen such that $u_1,u_{2,\delta}>0$ in $B_{\rho}$, in view of \eqref{eq:main_1_3}.  In addition, by monotonicity of $g$ we get that $\V_0>0$ in $B_\rho\setminus\{0\}$.
	 Moreover, by following the very same argument as in the proof of \Cref{thm:main_1}, one can check that $\V_0$ satisfies the condition \eqref{eq:G}. By construction, we have that $w\in H^s_{0,\rad}(B_1)$ satisfies
	\begin{equation*}
		\begin{bvp}
			(-\Delta)^s w&=\V_0 w &&\text{in }B_{R_\delta}, \\
			w&=0 &&\text{in }\R^N\setminus B_1, \\
			w(0)&=0.
		\end{bvp}
	\end{equation*}
	Again with the very same argument as in the proof of \Cref{thm:main_1}, it is easy to see that for any open, non-empty $\Omega\sub B_1$, $w\not \equiv 0$ in $\Omega$. Finally,   by suitably scaling, i.e., by letting $u(x):=w(\rho x)$, $\V(x):=\rho
	^{2s}\V_0(\delta x)$, $R:=\rho^{-1}$,   we can conclude the proof.
\end{proof}

\textbf{Acknowledgments.}
The authors warmly thank Lorenzo Brasco and Veronica Felli for useful discussions and suggestions.
E.M. is supported by the MIUR- PRIN project 202244A7YL.
R.O. is supported by the European Research Council (ERC), through the European Union’s
Horizon 2020 project ERC VAREG - Variational approach to the regularity of the free boundaries (grant agreement No. 853404) and by the 2024 INdAM-GNAMPA project no. \texttt{CUP\_E53C23001670001}. R.O. also acknowledges the MIUR Excellence Department Project awarded to the Department of Mathematics, University of Pisa, CUP I57G22000700001. B.V. is supported by  the MIUR-PRIN project 2022SLTHCE. This manuscript reflects only the authors' views and opinions and the Ministry cannot be considered responsible for them. The authors are members of the GNAMPA group of the Istituto Nazionale di Alta Matematica.\\

\textbf{Data availability.} Data sharing not applicable to this article as no datasets were generated or analyzed during
the current study.\\\\

\noindent{\large{\textbf{Declarations}}}\\

\textbf{Conflict of interest.} On behalf of all authors, the corresponding author states that there is no Conflict of interest

\bibliographystyle{aomalpha}

\bibliography{biblio_mov}

\end{document}